\documentclass[a4paper,12pt]{article}

\usepackage{amsthm,amsmath,stmaryrd,bbm,hyperref,geometry,color}
\usepackage[utf8]{inputenc}
\usepackage{amssymb}
\usepackage[english]{babel}
\usepackage{graphicx}
\usepackage{amsfonts,amssymb}
\usepackage{verbatim}
\usepackage{enumitem}
\usepackage[all]{xy}

\newcommand{\po}{\left(}
\newcommand{\pf}{\right)}

\newcommand{\cco}{\llbracket}
\newcommand{\ccf}{\rrbracket}
\newcommand{\R}{\mathbb R} 
\newcommand{\T}{\mathbb T} 
\newcommand{\Z}{\mathbb Z}

\newcommand{\dd}{\text{d}}
\newcommand{\na}{\nabla}

\newcommand{\nv}[1]{#1}

\newtheorem{thm}{Theorem}
\newtheorem{assu}{Assumption}
\newtheorem{lem}[thm]{Lemma}

\newtheorem{prop}[thm]{Proposition}

\author{Pierre Monmarch\'e}
\title{A note on Fisher Information hypocoercive decay for the linear Boltzmann equation.}

\begin{document}

\maketitle

\begin{abstract}
This note deals with the linear Boltzmann equation in the non-compact setting with a confining potential which is close to quadratic. We prove that in this situation, starting from a smooth initial datum, the Fisher Information (and hence, the relative entropy) with respect to the stationary state  converges exponentially fast to zero.
\end{abstract}

\section{Introduction}\label{SectionSettings}
We are interested in the long-time convergence to equilibrium of the solution $f$ of the so-called linear Boltzmann (or BGK) equation
\begin{equation}\label{eq:f}
\partial_t f_t(x,y) \ = \ -v\cdot \na_x f_t(x,y) + \na U(x) \cdot \na_y f_t(x,y) + \lambda Qf_t(x,y) 
\end{equation}
where $(x,y)\in\R^{2d}$, $d\in\mathbb N_*$, $\lambda>0$ is constant, $  U\in \mathcal C^2\po \R^d,\R\pf$ and $Q$ is either $Q_1$ or $Q_2$ with 
\begin{eqnarray*}
Q_1 f(x,y) & =& \gamma_{d,\sigma}(y) \int_{\R^d}   f(x,v)   \dd v - f(x,y)\\
Q_2 f(x,y) & =&   \sum_{k=1}^d \po \gamma_{1,\sigma}(y_k)\int_{\R} f(x,y_1,\dots,y_{k-1},w,y_{k+1},\dots,y_d) \dd w  - f(x,y)\pf 
\end{eqnarray*}
with,  for some  $\sigma>0$,
\[\gamma_{p,\sigma}(y) = \frac{e^{-\frac1{2\sigma^2}|y|^2} }{\po 2\pi \sigma^2\pf^{p/2} }\]
a Gaussian measure on $\R^p$. We assume that $f_0$ is a probability density so that, mass and positivity being conserved through time, $f_t$ is a probability density for all $t\geqslant 0$. Denoting $H(x,y)=U(x)+|y|^2/2$, we suppose that $\exp(-H/\sigma^2)$ is integrable and we denote by $\mu$ the probability law with density proportional to it (we also write $\mu$ this density). Then $\mu$ is a fixed point of \eqref{eq:f}. Our goal is to give a quantitative estimate for the convergence of a solution of \eqref{eq:f} toward $\mu$. In fact we will rather work with the relative density $h_t = f_t/\mu$, which solves 
\begin{equation}\label{EqPrincipale}
\partial_t h_t \ =\  Lh_t
\end{equation}
with
\[Lh(x,y)\ = \ -v\cdot \na_x h(x,y) + \na U(x) \cdot \na_y h(x,y) + \eta \po Ph - h\pf\,,\]
where $(P,\eta)$ is either $(P_1,\eta_1)$ or $(P_2,\eta_2)$ with $\eta_1=\lambda $, $\eta_2=\lambda  d$ and
\begin{eqnarray*}
P_1 h(x,y) & =& \int_{\R^d} h(x,v) \gamma_{d,\sigma}(v) \dd v\\
P_2 h(x,y) & =& \frac1d\sum_{k=1}^d\int_{\R} h(x,y_1,\dots,y_{k-1},w,y_{k+1},\dots,y_d) \gamma_{1,\sigma}(w) \dd w \,.
\end{eqnarray*}
Remark that $P_1$ and $P_2$ are Markov operators.

Equation~\eqref{eq:f} is a classical model in statistical physics, modelling the motion of a particle influenced by an external potential $U$ and by random collisions with other particles with Gaussian velocities. We refer the interested reader to \cite{Leautaud2015} and references within for details. Moreover, it intervenes in Markov Chain Monte Carlo methods. More precisely, denote $L^*$ the dual of $L$ in $L^2(\mu)$. Integrating by parts, we see that
\begin{equation}\label{eq:L*}
L^* \varphi(x,y) \ = \ v\cdot \na_x \varphi(x,y) - \na U(x) \cdot \na_y \varphi(x,y) + \eta \po P\varphi - \varphi\pf\,.
\end{equation}
 This is the generator of a Markov process $(X,Y)$  whose law solves \eqref{eq:f}. When $Q=Q_1$, the dynamics of the process is the following: the particle follows the Hamiltonian flow $\dot x = y$, $\dot y = -\na_x U(x)$ and, at random times with exponential law of intensity $\lambda$, the velocity $y$ is refreshed to a new Gaussian value. The motion is similar when $Q=Q_2$, except that each coordinate of the velocity has its own exponential clock, and is refreshed to a new Gaussian value independently from the other components. \nv{This process is sometimes called the randomized Hamiltonian Monte Carlo process \cite{RHMC}.} Since its law converges to $\mu$,  ergodic averages of the process can  be used as estimators for the expectations of some observables with respect to $\mu$. A non-asymptotic, quantitative long-time convergence estimate for \eqref{EqPrincipale} then classically provides bounds on the bias and variance of such estimators.

The question of the long-time convergence of \eqref{eq:f} (or equivalenty \eqref{EqPrincipale}) has been studied in much general forms in a number of works (see e.g. \cite{Leautaud2015} and references within).  The exponential convergence in the $L^2$ sense, i.e. the existence of constants $C,\theta>0$ such that
\begin{eqnarray*}
\| h_t - 1 \|^2_{L^2\po\mu\pf} & \leqslant & C e^{-\theta t} \| h_0 - 1 \|^2_{L^2\po\mu\pf},
\end{eqnarray*}
has been established under different assumptions by several authors (\cite{Herau2005,DMS2011,Leautaud2015,Achleitner2015}). This long-time convergence is said to be hypocoercive (\cite{Villani2009,DMS2011}), in the sense that $C$ is necessarily greater than 1 or, in other words, $h_t$ converges exponentially fast to $0$ but not at a constant rate (note that both the $L^2$ norm and the relative entropy studied below are non-increasing with time).

When one studies a system of $N$ particles with chain or mean-field interactions (so that $d=Nd'$, where $d'$ is the dimension of the ambient space), the $L^2$-norm is not well-adapted, since it scales badly in $N$. In these contexts, a more natural way to quantify the distance to equilibrium is the relative entropy $\int h \ln h \dd \mu$, as in \cite{MonmarcheVFP,MonmarcheGuillinVFP,OllaLetizia}. Nevertheless, entropic hypocoercivity results (see e.g. \cite{Villani2009,MonmarcheGamma,Wang2015}) are usually restricted to diffusion processes (i.e. differential operators). Indeed, since non-local operators such as $ \eta\po P-I\pf$ do not satisfy the chain rule, it is less easy to handle derivatives of non-quadratic quantities of $h$ and $\na h$. This is a general and important problem, related to the question of giving good definitions of non-local Fisher Information.

Nevertheless, for the linear Boltzmann equation, this  has been achieved by Evans \cite{Evans2017} in a recent paper in the case of the periodic torus (namely $x\in\T^d$, $\T = \R/\Z$) with no potential ($U=0$). The purpose of the present note is to show that the computations of Evans, together with the recent results on generalized Ornstein-Uhlenbeck processes (\cite{Arnold2015,Arnold2014,MonmarcheGamma}), allows in fact to deal with the case where $x\in\R^d$ and $U$ is close to being quadratic.

\begin{assu}\label{Hypothese}
\nv{There exist $K,\kappa>0$ such that $\kappa \leqslant \na^2 U(x) \leqslant K$ for all $x\in\R^d$.}
\end{assu}
In fact, we won't deal with the entropy itself, but with the classical Fisher Information
\begin{eqnarray*}
\mathcal I\po h\pf & =& \int \frac{|\na h|^2}{h} \dd\mu.
\end{eqnarray*}
Under Assumption~\ref{Hypothese}, the Hamiltonian $H$ is strictly convex, so that, by classical arguments (see e.g. \cite{BakryGentilLedoux}), $\mu$ satisfies a log-Sobolev inequality
\begin{eqnarray*}
\int h \ln h \dd \mu & \leqslant & c \mathcal I(h),
\end{eqnarray*}
where the constant $c$ only depends on $\kappa$ and $\sigma$. For elliptic or hypoelliptic diffusions, such as the kinetic Langevin (or Fokker-Planck) equation, a short-time regularization occurs, so that the Fisher Information is finite for all positive times given that the initial entropy is finite (see for instance \cite[Theorem 9]{MonmarcheGamma}). However, this is not the case for  equation \eqref{EqPrincipale}, and thus we will only consider smooth initial datum with $\mathcal I(h_0)<\infty$. More precisely, for the sake of simplicity, we will assume that, for some $\varepsilon\geqslant 0$,
\[h_0\in \mathcal A_\varepsilon \ :=\ \left\{g\in \mathcal C^\infty\po\R^{2d}\pf,\ \int g \dd \mu = 1, \ g\geqslant \varepsilon,\ \partial^\alpha g \text{ bounded }\forall \alpha\text{ multi-index}\right \}.\]
Note that, for any $\varepsilon\geqslant0$, the set $\mathcal A_\varepsilon$ is fixed by Equation \eqref{EqPrincipale}, as proved in \cite[Appendix]{CaceresCarilloGoudon} (here we don't need uniform in time estimates for the bounds of the derivatives). The result can then be extended by a density argument to all positive $h_0$ with $\mathcal I(h_0)<+\infty$. 
\begin{thm}\label{ThmMain}
Under Assumption \ref{Hypothese} let  $(h_t)_{t\geqslant 0}$ be a solution of \eqref{EqPrincipale} with  $h_0\in\mathcal A_0$ such that $\mathcal I(h_0) < +\infty$. Let
\[\theta \ = \  \frac23\po \frac{2\lambda K}{4K + \lambda^2} - \frac{(K-\kappa)^2}{\sqrt K } \pf \,,\qquad 
C\ = \  3 \max \po  \frac{4K^2}{4K + \lambda^2}, \frac{4K + \lambda^2}{4K^2}\pf\,.\]
Suppose that $\theta\geqslant 0$. Then, for all $t\geqslant 0$,
\begin{eqnarray*}
\mathcal I\po h_t\pf & \leqslant & C e^{-\theta t} \mathcal I\po h_0\pf.
\end{eqnarray*}
\end{thm}

\textbf{Remarks}
\begin{itemize}
\item The result is the same for $Q=Q_1$ or $Q_2$, and $C$ and $\theta$ do not depend on $\sigma$.
\item The log-Sobolev inequality satisfied by $\mu$ and Theorem~\ref{ThmMain} imply that, for some $C'>0$,
\begin{eqnarray*}
\int h_t \ln h_t \dd \mu & \leqslant & C' e^{-\theta  t}  \mathcal I(h_0).
\end{eqnarray*}
By the Pinsker's inequality, considering the Markov process $(X,Y)$ with generator $L^*$  given by \eqref{eq:L*}, we get that for all measurable set $A\subset \R^{2d}$,
\[|\mathbb P \po (X_t,Y_t)\in A\pf - \mu(A)| \ \leqslant \ \frac12 \|f_t -\mu\|_1 \ \leqslant \ \sqrt{\frac12 \int h_t \ln h_t \dd \mu  }\,. \]
This gives a bound on the bias of the Monte Carlo estimator of $\mu(A)$ based on the process $(X,Y)$, with constants $C'$ and $\theta$ that depends (explicitly) only on $\kappa,K,\gamma,\sigma$.

\item The assumption that the potential is convex is usual in the studies of the long-time behaviour of Markov processes. The fact that its Hessian is also bounded above, and more precisely that the Hessian is not too far from a constant matrix, is a much more rigid assumption, which already appeared in similar works \cite{Baudoin,Arnold2015}. Besides, there are examples of kinetic processes with a convex potential  with an unbounded Hessian, which does not converge exponentially fast to their equilibrium \cite{HairerMattingly}. Essentially,  we are able to deal with the Gaussian case because of some nice algebra, and have some room for a small perturbation.  More precisely, in the Gaussian case, the Jacobian of the drift is a constant matrix, so that the question of the contraction of a suitably modified Fisher Information boils down to a linear algebra problem (see Proposition~\ref{PropDeriveFisher} below). Then a Lipschitz perturbation from this linear case can be absorbed by the positive contraction of the linear case (see the end of the proof of Theorem~\ref{ThmMain}).
\item The rate of convergence is of order $\lambda$ when $\lambda$ goes to zero and of order $\lambda^{-1}$ when $\lambda$ goes to infinity, which is similar to the kinetic Langevin case (\cite{Iacobucci2017}), and expected. Indeed, when $\lambda$ is small, the typical time for the velocity to be refreshed (and thus, to mix) is $\lambda^{-1}$. On the other hand, when $\lambda$ is large, in a time of order 1, there are many jumps, and by the law of large number, the effective velocity is close to zero, and the position moves (and thus, mixes) slowly. If time is accelerated by a factor $\lambda$, the position then converges to an overdamped Langevin process.
\item \nv{In this particular close-to-quadratic case, Theorem \ref{ThmMain} answers the question raised in \cite[Section 1.5]{Evans2017}.}
\item Consider the case where $Q=Q_2$ and $U(x)=a|x|^2+ \frac1d\sum_{i,j=1}^d W(x_i-x_j)$ with an even potential $W$ with bounded Hessian and $a>0$.   This corresponds to a mean-field interaction between $N=d$ particles. Provided $\|W''\|_\infty$ is sufficiently small with respect to $a$ and $\lambda$, Theorem \ref{ThmMain} yields a speed of convergence to equilibrium wich is independent from the number of particles. Then, the arguments from \cite{MonmarcheVFP,MonmarcheGuillinVFP} may be adapted (the parallel coupling with Wiener processes being replaced by a parallel coupling with Poisson processes) to obtain uniform in time propagation of chaos, and long-time convergence for the non-linear PDE obtained at the limit (note that the latter is not the Boltzmann equation, for which the interaction lies at the level of the collisions rather than of the Hamiltonian).


\end{itemize}


 \section{Proof}\label{SectionProofs}

\nv{We write $h_t = e^{tL}h_0$  the solution of \eqref{EqPrincipale}  with initial condition $h_0$.} In the rest of the paper, we will always consider $h \in \mathcal A_\varepsilon$ with $\varepsilon>0$. Indeed, suppose that Theorem~\ref{ThmMain} has been proved for $h \in \mathcal A_\varepsilon$ with any arbitrary $\varepsilon>0$, and consider $h_0 \in \mathcal A_0$. Set $h_0^{(\varepsilon)} = (1-\varepsilon) h_0 + \varepsilon$. Then  $h_t^{(\varepsilon)} := e^{tL}h_0^{(\varepsilon)}= (1-\varepsilon) h_t + \varepsilon$ so that, applying Theorem~\ref{ThmMain} to $h_t^{(\varepsilon)}$ and letting $\varepsilon$ go to 0, the monotone convergence theorem yields the result for $h_t$. The restriction to the cases $\varepsilon>0$ will ensure that all the forthcoming derivations under the integral sign are correct. In particular, $\mathcal I(h)<+\infty$ for all $h\in\mathcal A_\varepsilon$ for $\varepsilon>0$.

  
We start with a general computation. Denoting by $A^T$ the transpose of a matrix (and seeing vectors as column matrices, so that the scalar product between two vectors $u$ and $v$ can be denoted by $u^T v$), for a symmetric matrix $M$, we write
 \begin{eqnarray*}
\mathcal I_M\po h\pf & =& \int \frac{(\na h)^T M \na h}{h} \dd\mu.
\end{eqnarray*}
Our aim is to construct $M$ such that $\partial_t \po \mathcal I_M\po e^{tL}h\pf\pf \leqslant - \theta \mathcal I_M\po e^{tL} h\pf$ with $\theta>0$. \nv{In the following, in a $2d \times 2d$ matrix, a $d\times d$ block equal to $\alpha I_d$ for some $\alpha \in \R$ will sometimes be denoted only  by $\alpha$, and $N\geqslant M$ stands for the usual order for symmetric matrices $N,M$.}

For an operator $A$, we write $\po \partial_t\pf_{|A}$ the derivative at $t=0$ along the semi-group $e^{tA}$.
 \begin{lem}\label{LemFischerSaut}
 Let $P$ be a Markov operator which fixes $\mathcal A_\varepsilon$,   $h\in\mathcal A_\varepsilon$  and $M=R^TR$ be a positive symmetric matrix. Then
 \begin{eqnarray*}
  \po \partial_t\pf_{|P-I} \mathcal I_{M}\po h\pf & \leqslant & -  \mathcal I_M\po h\pf + \mathcal I_M\po P h\pf  \,.
  \end{eqnarray*}
 \end{lem}
\begin{proof}
The computation is similar to \cite[Lemma 3]{Evans2017}. Indeed,
 \begin{eqnarray*}
  \po \partial_t\pf_{|P-I} \mathcal I_M\po h\pf & = & \int \frac{2\po \na h\pf^T R^T R \na \po Ph - h\pf}{h} - \frac{|R\na h|^2(Ph-h) }{h^2} \dd\mu\\
  & = &    \int -\frac{|R\na h|^2}{h}\po 1 +\frac{Ph}{h}\pf + 2\frac{(\na h)^T R^T R \na P h }{h} \dd\mu\\
  & =& -   \mathcal I_M\po h\pf +\lambda \int -\left|\frac{R\na h}{h} -\frac{R\na P h}{P h} \right|^2 Q h + \frac{|R\na P h|^2}{P h} \dd \mu\\
  & \leqslant & - \mathcal I_M\po h\pf + \mathcal I_M\po P h\pf  
  \end{eqnarray*}
  where we used the positivity of the density $h$.
\end{proof}

For $k\in\cco 1,d\ccf$, let $E_k$ be the $2d\times2d$ diagonal matrix with all its coefficients being zero except \nv{the $(d+k,d+k)^{th}$ being equal to 1, and
\[ E\ =\  \sum_{k=1}^d E_k \ = \  \begin{pmatrix}
  0 & 0\\ 0 & I_d
  \end{pmatrix}\,,\qquad E' \ = \ I_{2d} - E \ = \  \begin{pmatrix}
  I_d & 0\\ 0 & 0
  \end{pmatrix}\,.\]
In the particular case of \eqref{EqPrincipale}, Lemma~\ref{LemFischerSaut} yields the following. 

\begin{lem}\label{Lem:nouveau}
Let $\lambda>0$,  $h\in\mathcal A_\varepsilon$  and $M=R^TR$ be a positive symmetric matrix. Then,
\begin{eqnarray}\label{EqQ1}
  \po \partial_t\pf_{|\eta(P-I)} \mathcal I_{M}\po h\pf & \leqslant & -\lambda \po \mathcal I_{E M + M E  -E  M E }(h)\pf
\end{eqnarray}
for $(P,\eta)=(P_1,\eta_1)$. Moreover, this is also true for $(P,\eta)=(P_2,\eta_2)$ if the  right down $d\times d$ corner of $M$ is an homothety.
\end{lem}
}
  \begin{proof}
  We recall the following argument from \cite[Lemma 1]{Evans2017}. From $\na_y P_1 =0$ and $\na_x P_1 = P_1\na_x$, $\mathcal I_M\po P_1 h\pf = \int  \phi \po P_1(\na h,h)\pf \dd \mu$, where $\phi(u,v) = (u^T E' M E' u)/v$.   Applying Jensen's Inequality to the convex function $\phi$  and the Markov operator $P_1$, we get $\phi \po P_1(h,\na h)\pf \leqslant P_1 \phi(h,\na h)$. Integrated with respect to $\mu$ (which is fixed by $P_1$), this reads
\[ \mathcal I_M\po P_1 h\pf \ \leqslant \ \mathcal I_{E'ME'}\po   h\pf. \]
Applying Lemma \ref{LemFischerSaut} yields \eqref{EqQ1} (since $\eta_1=\lambda$).

Similarly, denoting
\begin{eqnarray*}
P_{2,k} f(x,y) & =&  \int f(x,y_1,\dots,y_{k-1},w,y_{k+1},\dots,y_d) \frac{e^{-\frac1{2\sigma^2} w^2}}{  \sigma\sqrt{2\pi} } \dd w, 
\end{eqnarray*}
for $k\in\cco 1,d\ccf$, we get with the previous argument
\[ \mathcal I_M\po P_{2,k} h\pf \leqslant \ \mathcal I_{(I_{2d}-E_k)M(I_{2d}-E_k)}\po   h\pf, \]
so that
\begin{eqnarray*}
\po \partial_t\pf_{|\eta_2 (P_{2}-I)} \mathcal I_{M}\po h\pf  &= &\sum_{k=1}^d  \po \partial_t\pf_{|\lambda(P_{2,k}-I)} \mathcal I_{M}\po h\pf \notag\\
  & \leqslant & -\lambda \po \mathcal I_{ \sum_{k=1}^d \po E_k   M + M   E_k- E_k  M  E_k \pf}(h)\pf.
\end{eqnarray*}
Now, suppose that the right down $d\times d$ corner of $M$ is an homothety, i.e. that
\[M\ = \ \begin{pmatrix}
M_{1} & M_{2}\\
M_2^T & \alpha I_d
\end{pmatrix}\]
for some matrices $M_i$ and some $\alpha>0$. In that case,
\[E M E \ = \ \alpha \begin{pmatrix}
0 & 0 \\ 0 & 1
\end{pmatrix} \ = \ \alpha \sum_{k=1}^d E_k E_k \ =\  \sum_{k=1}^d E_k M E_k,    \]
which means that we have obtained the same bound \eqref{EqQ1} on $\po \partial_t\pf_{|\eta_i(P_{i}-I)} \mathcal I_{M}\po h\pf$ for both $i=1,2$.
  \end{proof}
  
\bigskip

On the other hand, the derivative of $\mathcal I_M$ along the transport semi-group $e^{tA}$ where $A=y\cdot \na_x  - \na_x U(x) \cdot \na_y $ is  a classical computation (see e.g. \cite[Example 8]{MonmarcheGamma}), which we recall for the sake of completeness:
\begin{lem}\label{LemDeriveT}
For $h\in \mathcal A_\varepsilon$,
 \begin{eqnarray}\label{EqFisherDrift}
  \po \partial_t\pf_{|A } \mathcal I_M\po h\pf 
   & = & \int \frac{(\na h)^T (MJ + J^T M) \na h }{h} \dd \mu
  \end{eqnarray}
with $J = \begin{pmatrix}
 0 & -\na^2 U   \\ 1 & 0
\end{pmatrix}$.
\end{lem}

\begin{proof}
Since $A$ satisfies the chain rule,
 \begin{eqnarray*}
  \po \partial_t\pf_{|A } \mathcal I_M\po h\pf & = & \int \po \frac{2(\na h)^T M \na A h }{h} - \frac{Ah}{h^2} (\na h)^T M \na  h  \pf\dd \mu\\
   & = &  \int \po  A \po  \frac{2(\na h)^T M \na h }{h} \pf + \frac{2(\na h)^T M (\na A h-A\na h) }{h}\pf \dd \mu
  \end{eqnarray*}
where $A\na h$ should be understood coordinate by coordinate. Conclusion follows from $\int Ag \dd\mu = 0$ for all $g$ and $\na A h-A\na h = J\na h$.
\end{proof}

\nv{Combining the two previous results, we get:
\begin{prop}\label{PropDeriveFisher}
\nv{Under Assumption~\ref{Hypothese}}, let $h_t = e^{tL} h_0$ where $h_0 \in \mathcal A_{\varepsilon}$. Suppose that there exist $a,b,\theta\in \R$ with $b^2 < a$ and such that, for  $\xi\in\{\kappa,K\}$,
\begin{equation}\label{eq:lemMN}
N(\xi) \ := \ \nv{\begin{pmatrix}
2b  & \ &   a-\xi + \lambda b \\   a-\xi + \lambda b  & \ &   -2b\xi + \lambda a 
\end{pmatrix}} \ \geqslant \ \theta \begin{pmatrix}
1 & b  \\ b   & a   
\end{pmatrix} \ := \ \theta M.
\end{equation} 
Then, for all $t\geqslant0$,
 \begin{eqnarray*}
    \mathcal I_M\po h_t\pf & \leqslant &  e^{-\theta t} \mathcal I_M\po h_0\pf\,. 
  \end{eqnarray*}
\end{prop}
}
\begin{proof}
\nv{Let $a,b,\theta\in\R$ and $M$ be as in the proposition. Since $L=-A+\eta(P-I)$, Lemmas \ref{Lem:nouveau} and \ref{LemDeriveT} read
 \begin{eqnarray*}
   \partial_t \mathcal I_M\po h_t\pf & =&  \po \partial_t\pf_{|\eta(P-I)} \mathcal I_M\po h_t\pf  -  \po \partial_t\pf_{|A} \mathcal I_M\po h\pf \ \leqslant \  - \mathcal I_{N'}\po h\pf 
  \end{eqnarray*}
  with, for all $x\in \R^d$,
  \begin{eqnarray*}
  N'(x) & =& M \begin{pmatrix}  0 & \ & -\na^2 U(x)\\ 1 & \ &  \lambda  \end{pmatrix} + \begin{pmatrix}  0 & \ & 1\\ -\na^2 U(x) &\ &  \lambda \end{pmatrix}  M  - \lambda \begin{pmatrix}  0 &0\\ 0 & a \end{pmatrix} \\
&  = &
 \begin{pmatrix}
2b  & \ &   a-\na^2 U(x) + \lambda b \\   a-\na^2 U(x) + \lambda b  & \ &   -2b\na^2 U(x) + \lambda a 
\end{pmatrix}\,.
  \end{eqnarray*}
  The proof will be concluded (by the Gronwall's Lemma) if we prove that $N'(x) \geqslant \theta M$ for all $x\in \R^d$. Fix $x\in \R^d$, and let $\mathcal O(x)$ be an orthonormal $d\times d$ matrix such that $\mathcal O^T(x) \na^2 U(x) \mathcal O(x)$ is diagonal. Let
  \[\mathcal O' \ = \ \begin{pmatrix}
  \mathcal O(x)  & 0 \\ 0 & \mathcal O(x)
  \end{pmatrix}\,.\]
  Notice that $N'(x) \geqslant \theta M$ if and only if $\mathcal O^TN'(x)\mathcal O \geqslant \theta M$, where we used that $\mathcal O ^T M \mathcal O  =  M$. Now, $\mathcal O^TN'(x)\mathcal O \geqslant \theta M$ if and only if $N(\xi_k) \geqslant \theta M$ for all eigenvalues $\xi_k$ of $\na^2 U(x)$, $k\in\cco 1,d\ccf$. Writing such an eigenvalue   as $\xi_k = p_k \kappa + (1-p_k)K$ for some $p_k\in[0,1]$, we get
  \[N(\xi_k) \ = \ p_k N(\kappa) + (1-p_k) N(K) \ \geqslant \ \theta M\]
  by assumption, which concludes.}
\end{proof}

\nv{With Proposition~\ref{PropDeriveFisher} in hand, the proof of  our main result has boiled down to elementary computations.}

\begin{proof}[Proof of Theorem \ref{ThmMain}]
\nv{Let us find $a$ and $b$ such that $N(\xi)$ as given by \eqref{eq:lemMN} is definite positive for a given $\xi$ (to be chosen later on). For simplicity, we want to enforce the following conditions:
\[4b^2 \leqslant  a\,, \qquad a+\lambda b = \xi\,,  \qquad \lambda a \geqslant 4 b\xi\,,\]
which ensures that $N(\xi)$ is diagonal with positive terms and that the corresponding $M$ is definite positive. It is clear that such conditions are satisfied for $b$ small enough with $a=\xi-\lambda b$. More precisely, the first condition is implied by the third if $b\leqslant  \xi/\lambda$, and the third is implied by the second if $b\leqslant \lambda \xi /(4\xi + \lambda^2)$ (notice that $\lambda \xi /(4\xi + \lambda^2) \leqslant \xi/\lambda$). As a consequence, we chose
\[b \ = \  \frac{\lambda \xi}{4\xi + \lambda^2}\,, \qquad a \ =\  \xi - \lambda b \ = \ \frac{4\xi^2}{4\xi + \lambda^2}\,. \]
The condition $4b^2 \leqslant a$ is such that the corresponding matrix $M$ satisfies
\begin{equation*}
\frac12 
\begin{pmatrix}
1 & 0 \\ 0 & a 
\end{pmatrix}\ 
\leqslant \ M \ \leqslant \ \frac32 
\begin{pmatrix}
1 & 0 \\ 0 & a 
\end{pmatrix}\,.
\end{equation*}
The choice of $a$ and $b$ also ensures that
\[N(\xi) \ = \ \begin{pmatrix}
2b  & \ &   0 \\   0 & \ &   -2b\xi + \lambda a 
\end{pmatrix} \ \geqslant \ 
 \begin{pmatrix}
2b  & \ &   0 \\   0 & \ &   \frac{\lambda a}2 
\end{pmatrix}\,.
\]
For $\xi'\in\{\kappa,K\}$, for all $\gamma>0$,
\begin{align*}
N(\xi') \ & =\  N(\xi) + \begin{pmatrix}
0 & \ & \xi - \xi' \\ \xi-\xi' &\ &  2b(\xi-\xi')
\end{pmatrix} \\
& \geqslant \ 
 \begin{pmatrix}
2b  - \gamma(\xi-\xi')^2 & \ &   0 \\   0 & \ &   \frac{\lambda a}2  + 2b(\xi-\xi') - \frac{1}{\gamma}(\xi-\xi')^2
\end{pmatrix}\,.
\end{align*}
In other words,
\[N(\xi') \ \geqslant \  \theta_1  \begin{pmatrix}
1& \ &   0 \\   0 & \ &  a
\end{pmatrix}\]
with
\[\theta_1 \ = \ \min\po 2b \po 1 -\frac{\gamma}{2b}(\xi-\xi')^2 \pf ,  \frac\lambda2  \po 1 -\frac{2}{\lambda a\gamma }(\xi-\xi')^2 +   \frac{\xi-\xi'}{\xi}\pf \pf\,.\]
Using that $2b\leqslant \lambda/2$,  we chose $\gamma^2 = 4b/(\lambda a) = 1/ \xi$ to get that
\[\theta_1 \ \geqslant \ \theta_2 \ :=\  2b \po 1 -\frac{1}{2b\sqrt\xi }(\xi-\xi')^2 +   \frac{\xi-\xi'}{\xi}\pf\,.\]
Finally, we simply chose $\xi = K$, so that $\xi-\xi'\geqslant 0$ for $\xi'\in\{\kappa,K\}$. Assuming that $(K-\kappa)^2 \leqslant 2b\sqrt\xi$, we get that $\theta_2 \geqslant 0$ for both $\xi'\in \{\kappa,K\}$, and thus
\[N(\xi') \ \geqslant \  \theta_2  \begin{pmatrix}
1& \ &   0 \\   0 & \ &  a
\end{pmatrix} \ \geqslant \  \frac{2\theta_2}3 M \ = \ \theta M\,, \]
and we conclude by 
\[\mathcal I(h_t) \leqslant \frac{2}{\min(1,a)} \mathcal I_{M}(h_t) \leqslant \frac{2}{\min(1,a)} e^{-\theta t} \mathcal I_{M}(h_0) \leqslant \frac{3\max(1,a)}{\min(1,a)} e^{-\theta t} \mathcal I(h_0)\,.  \]
}
\end{proof}

\subsection*{Acknowledgements}
The author would like to thank Stefano Olla for introducing him to this question, and Max Fathi for fruitful discussions. He acknowledges financial support from  the project EFI ANR-17-CE40-0030 of the French National Research Agency.

\bibliographystyle{plain}
\bibliography{biblio}

\end{document}